\documentclass[a4paper,12pt]{article}

\usepackage[english]{babel}
\usepackage[utf8]{inputenc}
\usepackage{amsmath}
\usepackage{amsfonts}
\usepackage{graphicx}
\usepackage[colorinlistoftodos]{todonotes}
\usepackage{tikz-cd}
\usepackage{amsthm}
\usepackage{lipsum}

\theoremstyle{definition}
\newtheorem{definition}{Definition}[subsection]
\newtheorem{example}{Example}[subsection]
\newtheorem{remark}{Remark}[subsection]
\theoremstyle{plain}

\newtheorem{theorem}{Theorem}[subsection]
\newtheorem{lemma}{Lemma}[subsection]
\newtheorem{corollary}{Corollary}[subsection]
\DeclareMathOperator{\lt}{LT}
\DeclareMathOperator{\lc}{LC}

\DeclareMathOperator{\Ima}{Im}

\title{A case study of Gr\"{o}bner bases and Hilbert series}
\author{Soutrik Roy Chowdhury \footnote{Author is grateful to School of Mathematics, Trinity College Dublin for providing financial support to complete his masters research work there and Prof Vladimir Dotsenko for advising the thesis.}}
\newcommand{\Addresses}{{% additional braces for segregating \footnotesize
  \bigskip
  \footnotesize

   \textsc{School of Mathematics, Trinity College Dublin,
     Dublin 2}\par\nopagebreak
  \textit{E-mail address}-- \texttt{roychows@maths.tcd.ie}
}}
\date{}

\begin{document}
\maketitle

\begin{abstract}
In this expository writing I will give an introduction to Gr\"{o}bner bases and how to use it to compute Hilbert series from chains.
\end{abstract}
\tableofcontents
\section{Introduction}
 Suppose we have a polynomial ring (either commutative or non-commutative) in either one variable or multivariables. Also we are given an ideal $I$ of that. Let $f$ be a polynomial belonging to the ring. Now the question is whether it belongs to $I$ or not? When we have a single variable polynomial ring, we can use our known division algorithm to say whether $f$ belongs to $I$ or not. But in case of a multivariable polynomial ring the computation is very complex as we can not apply our division algorithm. Then the concept of Gr\"{o}bner bases easily solve this complexity.\\ 
 Next we are aware of the concepts of Hilbert series for a graded spaces/algebras. We will use the concept of Gr\"{o}bner bases to compute Hilbert series of some algebras through chains. This will lead us to find an exact sequence which can primarily satisfy our need but it can not tell us about the initial algebra. In the further work section we will mention about Anick's resolution which will fulfill the lack discussed in the previous section.\\ 
 So let us begin with the concept of Gr\"{o}bner bases:
 
\section{Background for Gr\"{o}bner Bases}

\subsection{Algebras}
\begin{definition}
An \textbf{algebra} is a vector space $V$(over a field $\mathbb{K}$)equipped with a multiplication $V \otimes V \rightarrow V$ with the following properties:
\begin{itemize}
\item $(x + y)z = xz + yz$ for $x,y,z \in V$
\item $x(y +z) = xy +xz$ for $x,y,z \in V$
\item $(ab)(xy) = (ax)(by)$ where $x,y \in V$ and $a,b \in \mathbb{K}$.
\end{itemize}
\end{definition}
\begin{example}
Algebra of polynomials $\mathbb{K}[x_1,x_2,\dots,x_n]$.
\end{example}
\begin{definition}
An algebra is \textbf{associative} if for the multiplication $\mu : V \otimes V \rightarrow V$ we have the equality
\begin{equation}
\mu(\mu \otimes id) = \mu( id \otimes \mu).
\end{equation}
That associative algebra is called \textbf{commutative} if for any $v_1,v_2 \in V$ we have
\begin{equation}
\mu(v_1,v_2) = \mu(v_2,v_1).
\end{equation}
It is \textbf{non-commutative} if the above equality fails for at least one pair of $v_1,v_2 \in V$.
\end{definition}
\begin{example}
Algebra of commutative polynomials with either single or mutlivariables over a field $\mathbb{K}$ is an example of commutative associative algebra.We usually denote it by $\mathbb{K}(x_1,x_2,\dots,x_n)$.\hspace{1mm}However algebra of non-commutative polynomials are examples of non-commutative associative algebra and we usually denote it by $\mathbb{K}\langle x_1,x_2,\dots,x_n\rangle$.
\end{example}

\subsection{Motivation}
Suppose we have a polynomial ring(either commutative or non-commutative) in one variable,\hspace{1mm}say $\mathbb{K}[x]$ where $\mathbb{K}$ is the ground field. Suppose we have an ideal $I$ of that ring $\mathbb{K}[x]$. Our job is to study the structure of $\mathbb{K}[x]/I$ in a constructive way. Later we will prove that the monomials not divisible by leading terms of the ideal $I$ form a basis of $\mathbb{K}[x_1,x_2,\dots,x_n]/I$. We will use such facts to get to know that given a polynomial $f \in \mathbb{K}[x]$,does it belong to the ideal $I$ or not? Now for single variable case it's easy as single variable polynomial ring $\mathbb{K}[x]$ is a Euclidean domain so we can perform Euclidean algorithm to know whether $f$ belongs to $I$ or not. But what for the case of multivariable polynomial ring $\mathbb{K}[x_1,x_2,\dots,x_n]$.\hspace{1mm}Suppose we have the same question: we have an ideal $I \subset \mathbb{K}[x_1,x_2,\dots,x_n]$ and we are given a polynomial $f \in \mathbb{K}[x_1,x_2,\dots,x_n]$ and are asked whether $f$ belongs to $I$ or not. Now here the case is difficult as multivariable polynomial ring  $\mathbb{K}[x_1,x_2,\dots,x_n]$ is neither a Euclidean domain nor a principal ideal domain.\\
\begin{lemma}
The polynomial ring  $\mathbb{K}[x_1,x_2,\dots,x_n]$ is not a principal ideal domain for $n >1$.
\end{lemma}
\begin{proof}
Take an ideal generated by $\{X_1,X_2\}$.\hspace{1mm}If $f$ generates this ideal, then $f$ divides both $X_1$ and $X_2$,\hspace{1mm}so $f$ is a constant term. So our ideal must be the entire ring. But $1$ is in the ring,\hspace{1mm}but not in the ideal. Contradiction.
\end{proof}
To solve this problem we have the concept of Gr\"{o}bner basis which is a type of basis defined carefully which tells that if we replace our generators $f_i$ of the ideal $I$ with a Gr\"{o}bner basis $g_j$ of the same ideal then we have the property that the remainder of $f$ on division by the polynomials $g_j$ is $0$ if and only if $f$ is in the ideal.\\
So we understand that to study the structure of $ \mathbb{K}[x_1,x_2,\dots,x_n]/I$ in a constructive way we require the concept of Gr\"{o}bner bases.The original definition was given in Bruno Buchberger's PhD thesis in 1965 [1]. Before moving to the definition of Gr\"{o}bner bases we require some preliminary materials:

\subsection{Preliminary materials}

\begin{theorem}[Hilbert basis theorem]
$I \subset \mathbb{K}[x_1,x_2,\dots,x_n]$ is always finitely generated,\hspace{1mm}so there exist $f_1,f_2,\dots,f_m \in \mathbb{K}[x_1,x_2,\dots,x_n]$ such that $I= \langle (f_1,f_2,\dots,f_m) \rangle$.\\
\end{theorem}

\begin{definition}
An \textbf{admissible ordering}\hspace{1mm} $"<"$\hspace{1mm} of monomials is a total ordering of all monomials in $\mathbb{K}[x_1,x_2,\dots,x_n]$ such that
\begin{itemize}
\item it is a well ordering i.e. there is no infinite decreasing sequences.
\item $m_1 < m_2 \Rightarrow m_1m_3 < m_2m_3 $ for any monomial $m_3$,\hspace{1mm}where $m_1,m_2 \in \mathbb{K}[x_1,x_2,\dots,x_n]$.
\end{itemize}
\end{definition}

\begin{lemma}
There is only one admissible ordering of monomials  in $\mathbb{K}[x]$ i.e.\\
\[ x^k < x^l\hspace{2mm} \text{if and only if}\hspace{2mm} k < l\]
\end{lemma}
\begin{proof}
Proof of this is easy. We will take an ordering like this,
\[ 1 < x\hspace{2mm}\text{implies}\hspace{2mm}x < x^2\hspace{2mm}\text{implies}\hspace{2mm}x^2 < x^3\hspace{2mm}\text{implies}\dots  \]
so we get an well ordering,\hspace{1mm}so this is the admissible ordering. Now suppose we take $x < 1\hspace{2mm}\text{implies}\hspace{2mm}x^2<x\hspace{2mm}\text{implies}\hspace{2mm}x^3 < x^2\hspace{2mm}\text{implies}\dots$,\hspace{1mm}then this implies an infinite decreasing sequence,\hspace{1mm}hence contradiction.
\end{proof}
\begin{remark}
For $n \geq 2$,\hspace{1mm} there are infinitely many admissible orderings.
\end{remark}
\begin{example}
\textbf{LEX}(lexicographic ordering),\hspace{1mm}it can be explained in this way,
\[ x_1^{i_1}x_2^{i_2}x_3^{i_3}\dots x_n^{i_n} < x_1^{j_1}x_2^{j_2}x_3^{j_3}\dots x_n^{j_n} \]
if
\[ i_1 < j_1 \hspace{1mm}\text{or} \hspace{1mm} \]
 \[  i_1=j_1 , i_2 < j_2 \hspace{1mm}\text{or} \hspace{1mm} \]
 \[  i_1=j_1 , i_2=j_2 , i_3 < j_3 \hspace{1mm} \text{or} \hspace{1mm}\]
  \[ \vdots \]
\end{example}
\begin{example}
\textbf{DEGLEX}(degree-lexicographic ordering),\\
a little difference with \textbf{LEX} is that here first we consider the degree then the \textbf{LEX} ordering. It can be explained as follows,
\[ x_1^{i_1}x_2^{i_2}x_3^{i_3}\dots x_n^{i_n} < x_1^{j_1}x_2^{j_2}x_3^{j_3}\dots x_n^{j_n}  \]
if
\[ i_1+i_2+i_3+\dots+i_n < j_1+j_2+j_3+\dots+j_n \hspace{1mm} \text{or}\]
\[ i_1+i_2+i_3+\dots+i_n = j_1+j_2+j_3+\dots+j_n \hspace{1mm} \text{and}\]
\[ x_1^{i_1}x_2^{i_2}x_3^{i_3}\dots x_n^{i_n}\hspace{1mm} <_\textbf{LEX}\hspace{1mm} x_1^{j_1}x_2^{j_2}x_3^{j_3}\dots x_n^{j_n}.\]
\end{example}
Let us fix an admissible ordering. Let $I \subset \mathbb{K}[x_1,x_2,\dots,x_n]$ be an ideal. From this $I$ we can find $\lt(I)$ (leading terms) or we can say $\lt(I)$\hspace{1mm}= space of linear combinations of monomials $m$ over $\mathbb{K}$ which are leading terms of elements of $I$. We say $m \in \lt(I)$ if $\exists \hspace{1mm}f \in I$ such that $f=cm + \sum c_im_i$,\hspace{1mm} where $m_i$'s are monomials with $m_i <m$ and $c_i \in \mathbb{K}$ and $c \neq 0$. However by $\lt(f)$ for $f \in I$ we will mean the leading term of the polynomial according to our fixed admissible ordering. For an example we fix an order $x > y$,\hspace{1mm}let our $f$ be $x^2+y^2$. Then $\lt(f)$ is $x^2$. Next by $\lc(f)$ we mean the co-efficient of leading term of $f$. We denote leading co-efficient by $\lc$.
\begin{lemma}
$\lt(I)$ is itself an ideal in $\mathbb{K}[x_1,x_2,\dots,x_n]$.
\end{lemma}
\begin{proof}
It is very easy to show. Let we take our previous $f$ i.e. $f=cm + \sum c_im_i$,\hspace{1mm}we form $m' = m''m$,\hspace{1mm}multiplying $m''$ with the equation of $f$ we get $fm'' = cm''m + \sum c_im''m_i$,\hspace{1mm}as $I$ is an ideal so $fm'' \in I$ implies $m' \in \lt(I)$.
\end{proof}

\begin{lemma}
Cosets of monomials $m \notin \lt(I)$ form a basis in $R = \mathbb{K}[x_1,x_2,\dots,x_n]/I$.
\end{lemma}
\begin{proof}
Let us first prove the linear independence.\\
Let $m_1,m_2,\dots,m_l \notin \lt(I)$,\hspace{1mm}without loss of generality assume $m_1 < m_2 < \dots < m_l$ \hspace{1mm}(where $<$ is our fixed admissible ordering),\hspace{1mm}then we have $c_1m_1 + c_2m_2 +\dots + c_lm_l = 0$ \hspace{1mm} in $R$,\hspace{1mm}where $c_i 's \in \mathbb{K}$. Let $f = c_1m_1 + c_2m_2 +\dots + c_lm_l \in I $. Then $\lt(f) \in \lt(I)$,\hspace{1mm} a contradiction unless $f = 0$ implies $c_1 = c_2 = \dots = c_l=0 $.\\
Next job is to show the spanning set property i.e. we need to show that if $m \in \lt(I)$,\hspace{1mm} then $m$ is a linear combination in $R$ of cosets of monomials not present in $\lt(I)$. We will prove this with the help of contradiction. Let's take the smallest $m \in \lt(I)$ for which such a combination doesn't exist. Now by definition, $\exists f \in I$,\hspace{1mm} such that
$0 = f = cm + \sum c_im_i$,\hspace{1mm}with $m_i < m$,\hspace{1mm}and each of $m_i's$ is not in $\lt(I)$. Then we have $m = -\sum \frac{c_i}{c}m_i$ which can be represented as a combination of cosets of elements outside $\lt(I)$. Contradiction.
\end{proof}
\section{ Gr\"{o}bner bases and Diamond lemma}
\subsection{Gr\"{o}bner basis}
\begin{definition}
$G \subset I$ is called a \textbf{Gr\"{o}bner basis} of $I$ if $\{\lt(g) |\hspace{1mm}g \in G \}$ generate the ideal $\lt(I)$ i.e. for each $f \in I$,\hspace{1mm}$\lt(f)$ is divisible by $\lt(g)$ for some $g \in G$.
\end{definition}
\begin{lemma}
$\langle G \rangle = I$.
\end{lemma}
\begin{proof}
We know that $(G) \subset I$, suppose assume that $(G) \neq I$. Let $f \in I\backslash(G)$ with smallest possible leading term. Then $\lt(f)= m \lt(g)$ for some $g \in G$,\hspace{1mm} $m \in I$. Let $F = f - \frac{\lc(f)}{\lc(g)}mg$,\hspace{1mm}where $\lc$ is the leading coefficient of leading term, then we have $\lt(F) < \lt(f)$,\hspace{1mm}it implies $F \in I \Rightarrow F \in (G)$,\hspace{1mm}then $f = F + \frac{\lc(f)}{\lc(g)}mg \in (G)$,\hspace{1mm}which is a contradiction.
\end{proof}
\begin{remark}
We have already proved that monomials not divisible by $\lt(G)$ form a basis of  $\mathbb{K}[x_1,x_2,\dots,x_n]/I$.
\end{remark}

We are now going to define \textit{reduction} and \textit{S-polynomial} for the commutative case, these two definitions play an important role for computation of Gr\"{o}bner bases.\hspace{1mm}For all these definitions we have $\mathbb{K}(x_1,x_2,\dots,x_n)$ as the commutative ring.
\begin{definition}
Suppose $f_1$,$f_2$ are two polynomials belong to our defined ring,\hspace{1mm}such that there exists a monomial $m$ with
\[ \lt(f_1) = m \lt(f_2).\]
Then
\begin{equation}
R_{f_2}(f_1) = f_1 - \frac{\lc(f_1)}{\lc(f_2)}mf_2
\end{equation}
is called \textbf{reduction} of $f_1$ with respect to $f_2$.
\end{definition}
\begin{definition}
We have two polynomials $f_1$ and $f_2$ in the ring,\hspace{1mm}suppose there exist monomials $m_1$,$m_2$ such that\\
\begin{equation} m_1\lt(f_2) = m_2\lt(f_1) \hspace{2mm}\text{and}\hspace{2mm} \deg(m_1) < \deg \lt(f_1) 
\end{equation}
then
\begin{equation}
 S(f_1,f_2) = \frac{1}{\lc(f_1)}m_2f_1 - \frac{1}{\lc(f_2)}m_1f_2
\end{equation}
is called \textbf{S-polynomial} with respect to a small common multiple (4).
\end{definition}
\begin{example}
Before giving the example, we would like to point out that throughout this paper $a \cdot b = ab$ for any $a,b$ belonging to either field,algebras etc. Sometime for our better understanding and to deal with some scenarios we use the multiplication symbol '$\cdot$'. \\
Let us give an example to show how reduction and S-polynomial work, suppose we have a commutative polynomial ring in 2 variable i.e. $\mathbb{K}(x,y)$.\hspace{1mm}We will pick \textbf{DEGLEX} ordering. Now let $f_1 = x^3 - y^2$ and $f_2 = x^3 -x +1$.\hspace{1mm}Then one can see for both $f_1$ and $f_2$, $\lt$ is $x^3$. So $\lt(f_1) = \lt(f_2) = x^3$ implies $m = 1$ so that $x^3 = 1\cdot x^3$,\hspace{1mm}our reduction of $f_1$ w.r.t $f_2$ will be then $x^3-y^2-(x^3-x+1) = x-y^2-1$.\\
And while computing S-polynomial with respect to a small common multiple we have 3 choices,
\[ 1 \cdot x^3 = 1 \cdot x^3\]
\[ x \cdot x^3 = x \cdot x^3 \]
\[ x^2 \cdot x^3 = x^2 \cdot x^3\]
hence for each cases we can compute the S-polynomial using our formula (3). For an example if we consider $x \cdot x^3 = x \cdot x^3  $ then our S-polynomial will be $x^2(x^3-y^2)-x^2(x^3-x+1) = x^3 - x^2y^2- x^2$.
\end{example}
\begin{remark}
In the commutative case S-polynomials for different common multiples are easily related, but in a non-commutative case, they all carry important information. We will see this when we will compute the Gr\"{o}bner basis for non-commutative case.
\end{remark}
\subsection{Diamond lemma}
\begin{lemma}
\textbf{Diamond lemma}:\\
$G \subset I$ forms a Gr\"{o}bner basis if and only if for each $g_1,g_2 \in G$\\
$R_{g_2}(g_1)$ (if defined) can be reduced to $0$ modulo $G$.\\
And also for each $g_1,g_2 \in G$ and each small common multiple of\hspace{1mm} $\lt(g_1),\lt(g_2)$;\hspace{1mm}the corresponding S-polynomial can be reduced to $0$ modulo $G$.
\end{lemma}
We will write another lemma which is equivalent to the \textit{diamond lemma}. We will prove this lemma only as it's easy to prove:
\begin{lemma}
Assume $\langle G \rangle = I$,\hspace{1mm}then the following statements are equivalent:\\
1. $G$ is a Gr\"{o}bner basis of $I$.\\
2. All reductions and all S-polynomials of pair of elements of $G$ can be reduced to $0$ modulo $G$.\\
3. For every $f \in I$,\hspace{1mm}$f$ admits a representation\\
\[ f=h_1g_1 + h_2g_2+\dots+h_ng_n \hspace{1mm};\hspace{1mm} g_i \in G\]
with \[ \lt(f) = \max(\lt(h_ig_i))\]
\end{lemma}
\begin{definition}
$f$ can be reduced to $0$ modulo $G$,\hspace{1mm}if there exists \hspace{1mm} $g_1,g_2,\dots,g_n \in G$ such that
\[ R_{g_m}(\dots\dots R_{g_2}(R_{g_1}(f))\dots) = 0\]
\end{definition}
\begin{proof}
We need to show that
\[ 1 \hspace{2mm}\text{implies}\hspace{2mm}2 \hspace{2mm}\text{implies}\hspace{2mm}3\hspace{2mm}\text{implies}\hspace{2mm} 1 \]
Now 3 implies 1 is trivial from the definition of Gr\"{o}bner bases. We will at first prove 1 implies 2.\\
Suppose $G$ is a Gr\"{o}bner basis,\\
then every $f \in I$ can be reduced to $0$ modulo $G$,\\
( $f \in I \Rightarrow \lt(f) = m\lt(g)$ for some $g \in G$,\hspace{1mm}then $R_g(f)$ has smaller leading term and we proceed with that and proceed further until get $0$.\\
Similarly for $f \in I$ we have $m_1\lt(f)=m_2\lt(g)$ for some $g \in G$,\hspace{1mm}with $\deg(m_1) < \deg \lt(g)$,\hspace{1mm}so we have $S(f,g)$ with smaller leading term and proceed like this until get $0$.)\\ 
This is exactly what statement (2) says.\\
Next we will show 2 implies 3.\\
As $I = (G)$,\hspace{1mm}lets take $f \in I$,
\[ f=h_1g_1 + h_2g_2 + \dots + h_ng_n; \hspace{2mm}\text{with}\hspace{1mm} \lt(f) < \max(\lt(h_ig_i)). \]
Our main objective is to show how to replace this combination by another one with smaller $\max(\lt(h_ig_i))$,\hspace{1mm}if still bigger than $\lt(f)$,\hspace{1mm}continue until it becomes $\lt(f)$. Without loss of generality,
\[ \lt(h_1g_1) = \lt(h_2g_2) = \dots  = \lt(h_kg_k) = \max \lt(h_ig_i)  \]
\[ and \hspace{1mm}\lt(h_pg_p) < \max \lt(h_ig_i)\hspace{2mm}\text{for} \hspace{1mm} p > k.  \]
We will use induction on $k$ to prove our desired result.\\
if $k = 1$,\hspace{1mm}a contradiction as we can't cancel,so $k \geq 2$.\hspace{1mm}Now
\[ \lt(h_1g_1) = \lt(h_1)\lt(g_1)   \]
\[ \lt(h_2g_2) = \lt(h_2)\lt(g_2).  \]
Without loss of generality assume these 3 conditions:\\
1. $\lt(g_1)$ is divisible by $\lt(g_2)$.\\
2. $\lt(g_1),\lt(g_2)$ have a small common multiple.\\
3. $\lt(g_1),\lt(g_2)$ have no common divisors.\\
Let us first deal with (1),
\[ \lt(g_1) = m\lt(g_2)  \]
\[ \text{then} \hspace{2mm} R_{g_2}(g_1) = g_1 - \frac{\lc(g_1)}{\lc(g_2)}mg_2  \]
\[ \text{now} \hspace{2mm}h_1g_1 + h_2g_2 = h_1\left( R_{g_2}(g_1) + \frac{\lc(g_1)}{\lc(g_2)}mg_2\right) + h_2g_2 \]
 \[                       = h_1R_{g_2}(g_1) + \left(\frac{\lc(g_1)}{\lc(g_2)}mh_1 + h_2\right)g_2  \]
 \[ \text{with} \hspace{3mm} \lt(R_{g_2}(g_1) < \lt(g_1)  \]
 \[ \text{so} \hspace{2mm} h_1R_{g_2}(g_1) = \sum_{i}\tilde{h_i}g_ih_1  \]
 \[ \lt(g_1) > \lt(R_{g_2}(g_1)) = \max ( \lt(\tilde{h_i}g_i)) \]
 So (1) replace the combination by another one with smaller $k$.\hspace{1mm}Lets move to case (2),
 \[ \lt(h_1)\lt(g_1) = \lt(h_2)\lt(g_2)  \]
 suppose that,
 \[ \lt(g_1) = m_1d  \]
 \[ \lt(g_2) = m_2d  \]
 where $m_1,m_2$ have no common factors,
 \[ \lt(h_1)m_1d = \lt(h_2)m_2d    \]
 canceling $d$ from both side we get,
 \[ \lt(h_1)=em_2  \]
 \[ \lt(h_2)=em_1,\hspace{3mm}\text{ for some}\hspace{1mm} e,  \]
 \[ h_1g_1 + h_2g_2 = \lc(h_1)\lt(h_1)g_1 + ( h_1 - \lc(h_1)\lt(h_1))g_1   \]
 \[ + \lc(h_2)\lt(h_2)g_2 +( h_2 - \lc(h_2)\lt(h_2))g_2  \]
 as\hspace{1mm} $( h_1 - \lc(h_1)\lt(h_1))g_1 +( h_2 - \lc(h_2)\lt(h_2))g_2$ \hspace{1mm} have smaller leading term so we will not consider this in our account and proceed with the remaining terms,\hspace{1mm}i.e. we deal with
 \begin{equation} \lc(h_1)\lt(h_1)g_1 + \lc(h_2)\lt(h_2)g_2 . \end{equation}
 We have
 \[ m_2\lt(g_1) = m_1\lt(g_2)  \]
 so the S-polynomial w.r.t small common multiple,
 \[ S = \frac{1}{\lc(g_1)}m_2g_1 - \frac{1}{\lc(g_2)}m_1g_2   \]
 Putting the values of $\lt(h_1)$ and $\lt(h_2)$ in eq (6) we get,
 \[ \lc(h_1)em_2g_1 + \lc(h_2)em_1g_2   \]
 \[ = \lc(h_1)\left(\lc(g_1)S + \frac{\lc(g_1)}{\lc(g_2)}m_1g_2 \right)e + \lc(h_2)em_1g_2 \]
 \[ = \lc(h_1)\lc(g_1)S + \left( \lc(h_1)\frac{\lc(g_1)}{\lc(g_2)}m_1e + \lc(h_2)m_1e \right)g_2  \]
 So we have again replaced the combination with a smaller $k$.\\
 Finally we have case (3) in hand,
 \[ \lt(h_1)\lt(g_1) = \lt(h_2)\lt(g_2)  \]
 \[ \Rightarrow \lt(h_1) = \lt(g_2)e   \]
 \[ \text{and} \hspace{3mm} \lt(h_2) = \lt(g_1)e, \hspace{2mm}\text{for some}\hspace{1mm}e \]
 Now
 \[ h_1g_1 + h_2g_2  \]
 \[ = \lc(h_1)\lt(h_1)g_1 + \lc(h_2)\lt(h_2)g_2 + \hspace{2mm}\text{lower terms}. \]
 Then proceed with 
 \[ \lc(h_1)\lt(h_1)g_1 + \lc(h_2)\lt(h_2)g_2   \]
 replacing the values of $\lt(h_1)$ and $\lt(h_2)$,
 \begin{equation}
 = \lc(h_1)\lt(g_2)eg_1 + \lc(h_2)\lt(g_1)eg_2.
 \end{equation}
 We have 
 \[ \lt(g_2) = \frac{1}{\lc(g_2)}(g_2 - \bar{g_2}),   \]
 replacing this in eq (7) we get
 \[ \frac{\lc(h_1)}{\lc(g_2)}(g_2 - \bar{g_2})eg_1 + \lc(h_2)\lt(g_1)eg_2   \]
 which again makes either $\lt$ on $k$ smaller.\\
 This is how we proceed using induction and prove statement (3) from (2).
\end{proof}
\section{Non-commutative Gr\"{o}bner bases}
We consider the non-commutative polynomial ring $\mathbb{K}\langle x_1,x_2,\dots,x_n \rangle$.
\begin{definition}
We have two polynomials $f,g$ belong to the above ring,with there exists monomials $m_1,m_2$ such that,
\[  \lt(f) = m_1\lt(g)m_2 \]
then
\begin{equation} R_g(f) = f - \frac{\lc(f)}{\lc(g)}m_1gm_2   \end{equation}
is called \textbf{reduction} of $f$ with respect to $g$.
\end{definition}
\begin{definition}
For $f$,$g$ of the ring and for any small common multiple of $\lt(f),\lt(g)$;\hspace{1mm}there exists two monomials $m_1,m_2$ with
\[   \lt(f)m_2 = m_1\lt(g) \hspace{2mm}with \hspace{1mm} \deg m_1 < \deg \lt(f) \]
then
\begin{equation}
S(f,g) = \frac{1}{\lc(f)}fm_2 - \frac{1}{\lc(g)}m_1g
\end{equation}
is called the \textbf{S-polynomial} with respect to small common multiples.
\end{definition}
\begin{remark}
$ I \subset \mathbb{K}\langle x_1,x_2,\dots,x_n \rangle$ is a two-sided ideal. Then $\lt(I)$ which is linear span of $\lt(f)$ with $f \in I$ is also a two-sided ideal.
\end{remark}

\begin{definition}
$ G \subset I$ is a \textbf{Gr\"{o}bner basis} of $I$ if for every $f \in I$,\hspace{1mm}we have $\lt(f) = m_1\lt(g)m_2$ for some $g \in G$.
\end{definition}
\begin{lemma}[Diamond lemma for non-commutative case]
$G \in I$ forms a Gr\"{o}bner basis of $I$ if and only if for each $g_1,g_2 \in G$;\hspace{1mm}$R_{g_2}(g_1)$(if defined) can be reduced to $0$ modulo $G$. Also for each $g_1,g_2 \in g$ and each small common multiple of $\lt(g_1),\lt(g_2)$;\hspace{1mm}the corresponding S-polynomial can be reduced to $0$ modulo $G$.
\end{lemma}
\begin{proof}
As the proof is similar like the commutative case so we will skip this.
\end{proof}

\section{Terminology for Gr\"{o}bner bases}
Before moving towards the computation of Gr\"{o}bner bases for both commutative and non-commutative cases we will first give some terminology regarding Gr\"{o}bner bases.
\begin{definition}
\textbf{Normal monomial:}\\
 Given a Gr\"{o}bner basis $G \subset I$,\hspace{1mm}normal monomials with respect to $G$ are those monomials which are not divisible by $\lt(g)$,\hspace{1mm}for $g \in G$.\\
 We sometime call normal monomials as normal words.
\end{definition}
\begin{lemma}
Cosets of normal monomials form a basis of $\mathbb{K}[x_1,x_2,\dots,x_n]/I$.
\end{lemma}
\begin{proof}
This is proved earlier in lemma 2.3.3
\end{proof}
\begin{definition}
\textbf{Reduced Gr\"{o}bner basis:}\\
$G$,\hspace{1mm}a Gr\"{o}bner basis of $I$ is reduced if for each $g \in G$,
\begin{itemize}
\item $\lc(g) = 1$.
\item $g - \lt(g)$ is a linear combination of normal monomials.
\end{itemize}
\end{definition}
\begin{theorem}
Let us fix an admissible ordering. Then every $I$ has a unique reduced Gr\"{o}bner basis.
\end{theorem}
\begin{proof}
Let us take some Gr\"{o}bner basis $G \subset I$.\\
First condition of reduced Gr\"{o}bner basis is easy to satisfy as we just divide each $g$ by it's $\lc$,\hspace{1mm}i.e. $ g \rightarrow g/\lc(g)$.\\
The reduction and S-polynomial suggests that remaining terms of $g$ is not divisible by the leading term of any terms in $G$ which implies that $g - \lt(g)$ is a linear combination of normal monomials. Now we will prove the uniqueness.\\
Let $\{ f_1,f_2,\dots,f_s \}$ and $\{g_1,g_2,\dots,g_s \}$ be two reduced and ordered Gr\"{o}bner bases so that $\lt(f_i)=\lt(g_i)$ for each $i$.\hspace{1mm}Consider $f_i - g_i \in I$,\hspace{1mm}if it's not $0$,\hspace{1mm}then its leading term must be a term that appeared either in $f_i$ or in $g_i$.\hspace{1mm}In either case,\hspace{1mm}this contradicts the fact that the bases being reduced,\hspace{1mm}so in fact we get our required $f_i = g_i$.
\end{proof}

\section{Computation of Gr\"{o}bner bases}
 In this section we will show how to compute Gr\"{o}bner basis for an ideal $I$ of a polynomial ring $\mathbb{K}[x_1,x_2,\dots,x_n]$. We will show for both commutative and non-commutative polynomial rings. So let us begin with a computation for a commutative polynomial ring. But at first we will show the general algorithm to compute the Gr\"{o}bner basis.
 
 \subsection{Buchberger's algorithm}
 We start with an ideal $I$ generated by a set $G$. The Buchberger's algorithm[2],\hspace{1mm}which is a simple consequence of lemma 3.1.2,\hspace{1mm}is the following:\\
 Step 1: If the leading term of any element $u$ of $G$ occurs inside the leading term of another element $v$ of $G$,\hspace{1mm}then we reduce $v$ by subtracting off the required multiple of $u$. In general we will perform the reduction mentioned in either definition 3.1.2 or definition 4.0.5.\\
 Step 2: For each pair of distinct elements of $G$ (for both commutative and non-commutative cases) or even in non-commutative case for a same element we compute the S-polynomial and a remainder of it.\\
 Step 3: If the remainder can be reduced further then we will follow step 1 or we will add that term in our set $G$. If all S-polynomials reduce to $0$,\hspace{1mm}then the algorithm ends and $G$ is the Gr\"{o}bner basis of $I$. If not then we will continue further with our 3 steps.\\
 For commutative cases the algorithm ends in a finite number of stages. However for a non-commutative case there is no guarantee of the termination of the algorithm after a finite number of stages. In that case we start adding all elements which can't be reduced further in our set $G$ and in most cases we have seen a combinatorial interpretation for our terms in $G$. 
 
\subsection{In case of a commutative polynomial rings}
\begin{example}
 We have previously defined what is meant by a commutative polynomial ring. Let us take $\mathbb{K}(x_1,x_2)$ as our commutative polynomial ring with two variables $x_1,x_2$. Suppose there are two polynomials
 \[  h_1(x_1,x_2) = x_1^2 + x_2^2  \]
 \[  h_2(x_1,x_2) = x_1^3 + x_2^3 \]
 belonging to our polynomial ring $\mathbb{K}(x_1,x_2)$. We will compute the Gr\"{o}bner basis for $I=(h_1,h_2) \subset \mathbb{K}(x_1,x_2)$.\\
 Let us fix an admissible ordering. Usually we take \textbf{DEGLEX} ordering. So here we consider $x_1 > x_2$. So we get $\lt(h_1)=x_1^2$ and $\lt(h_2)=x_1^3$. So initially our set is $G=\{ h_1,h_2 \}$. But we see that $h_2$ can be reduced further. So we have $x_1^3= x_1\cdot x_1^2$,
 \[ R_{h_1}(h_2) = (x_1^3 +x_2^3) - x_1(x_1^2 +x_2^2)  \]
 \[  = x_2^3 -x_1x_2^2.     \]
So we have obtained a new term $x_2^3 - x_1x_2^2$ which cannot be reduced further,\hspace{1mm}we add this to our set $G$ which is now $\{ h_1,R_{h_1}(h_2) \}$. We call $R_{h_1}(h_2)$ as $h_3$. We see that the leading term of $h_3$ is $x_1x_2^2$. We have also found that $ x_1\cdot x_1x_2^2 = x_1^2 \cdot x_2^2$. So we will compute the S-polynomial between $h_1,h_3$.
\[ S(h_1,h_3) = -x_1(x_2^3 - x_1x_2^2) - (x_1^2 + x_2^2)x_2^2  \]
\[  = -x_1x_2^3 - x_2^4.   \]
The term $-x_1x_2^3 - x_2^4$ has $x_1x_2^3$ as the leading term which can be reduced further through $\lt(h_3)$. We get $x_1x_2^3 = (x_1x_2^2)\cdot x_2$. hence the reduction yields
\[ -x_1x_2^3 - x_2^4 -(x_2^3 -x_1x_2^2)x_2   \]
\[ = -2x_2^4    \]
which cannot be reduced further and also one cannot compute more S-polynomial. Hence we add $-2x_2^4$ in our set $G$ and the final set $G$ is our Gr\"{o}bner basis for $I$,\hspace{1mm}the set is precisely as follows
\[ \{ x_1^2+x_2^2,\hspace{1mm}x_2^3 -x_1x_2^2,\hspace{1mm}-2x_2^4 \}.   \]
\end{example}
\begin{remark}
The reduced Gr\"{o}bner basis of $I$ of our previous example is given by $\{ x_1^2+x_2^2,\hspace{1mm}x_1x_2^2-x_2^3,\hspace{1mm}x_2^4 \}$. It is not very difficult to obtain this reduced Gr\"{o}bner basis from our computed Gr\"{o}bner basis. If we recall the definition of reduced Gr\"{o}bner basis we will see that all leading co-efficients of the reduced basis should be $1$. So we just divide terms $-2x_2^4, x_2^3 - x_1x_2^2$ of $\{ x_1^2+x_2^2,\hspace{1mm}x_2^3 -x_1x_2^2,\hspace{1mm}-2x_2^4 \}$ by $-2$ and $-1$ respectively to obtain $\{ x_1^2+x_2^2,\hspace{1mm}x_1x_2^2-x_2^3,\hspace{1mm}x_2^4 \}$. Indeed $x_2^2$ and $x_2^3$ are normal monomials. So we have obtained the reduced Gr\"{o}bner basis of $I$.  
\end{remark}

\subsection{In case of non-commutative polynomial rings}

\begin{example}
Let us consider $\mathbb{K}\langle x,y \rangle$ as a non-commutative polynomial ring. We are going to compute Gr\"{o}bner basis for $I= (x^2 -xy)$. So we begin with our set $G$ as $\{ x^2-xy \}$ whose leading term is $x^2$ (we consider the \textbf{DEGLEX} for this case i.e. here $x>y$). We need to compute S-polynomial between polynomial $f_1= x^2-xy$ and $f_2=x$. We have $x^2 \cdot x = x \cdot x^2$. Then
\[ S(f_1,f_2) = (x^2-xy)x - x(x^2-xy)   \]
\[ = xxy - xyx.   \]
We see that $xxy-xyx$ whose leading term is $xxy$ can be reduced further. So we have $xxy = x^2 \cdot y$. So the reduction is
\[ (xxy-xyx) - (x^2-xy)y  \]
\[ = xxy - xyx - x^2y + xyy    \]
\[ = xyy - xyx   \]
which cannot be reduced further. so we include $xyy-xyx$ in our set $G$ which is now $\{ x^2-xy,xyy-xyx \}$. We see that $x^2 \cdot yx = x \cdot xyx$,\hspace{1mm}hence we compute S-polynomial between those elements of our set $G$ and we get
\[ (x^2-xy)yx + x(xyy - xyx)  \]
\[ = x^2yx - xyyx + xxyy - xxyx   \]
\[ = xxyy- xyyx    \]
this element with leading term $xxyy$ can be reduced further and we get $xxyy = x^2 \cdot yy$. So the reduction gives
\[ (xxyy - xyyx) - (x^2 -xy)yy   \]
\[ = xxyy - xyyx - x^2yy + xyyy  \]
\[ = xyyy - xyyx \]
which cannot be reduced further and we add this term in our existing set $G$ and obtain $\{ x^2-xy,xyy-xyx,xyyy-xyyx \}$.\\
Now we claim that the Gr\"{o}bner basis for $I= (x^2-xy)$ is given by
\[ \{ x^2 - xy \} \cup \hspace{1mm} \bigcup_{i=2}^{\infty} \{ xy^i - xy^{i-1}x \}.  \]
This is indeed very easy to prove. We will prove it by method of induction. We have already shown for $i=2,3$. Suppose upto $i=k$ steps the Gr\"{o}bner basis of $I$ is
\[ \{ x^2 - xy \} \cup \hspace{1mm} \bigcup_{i=2}^{k} \{ xy^i - xy^{i-1}x \}.  \]
But none of $\{xy^i - xy^{i-1}x \}$ for $2 \leq i \leq k$ have S-polynomial between them,\hspace{1mm}so $\lt(xy^k-xy^{k-1}x)= xy^{k-1}x$ and we have $x^2 \cdot y^{k-1}x = x \cdot xy^{k-1}x$ which leads us to compute the s-polynomial between $x^2-xy$ and $xy^k-xy^{k-1}x$ and we get
\[ (x^2-xy)y^{k-1}x + x(xy^k - xy^{k-1}x)   \]
\[ = x^2y^{k-1}x - xy^kx + x^2y^k - x^2y^{k-1}x    \]
\[ = x^2y^k - xy^kx    \]
which has the leading term $x^2y^k$ which can be reduced further. We get $x^2y^k = x^2 \cdot y^k$. Hence the reduction is 
\[ (x^2y^k - xy^kx) - (x^2 -xy)y^k  \]
\[ = xy^{k+1} - xy^kx   \]
which cannot be reduced further and so we add the term $xy^{k+1}- xy^kx$ in our existing set $G$.\\
Hence the proof and we have obtain our said Gr\"{o}bner basis.
\end{example}
 
\begin{example}
Let us give another example to compute Gr\"{o}bner basis for a non-commutative polynomial ring. Consider $\mathbb{K}\langle x,y,z \rangle$ as our non-commutative polynomial ring. We are going to compute Gr\"{o}bner basis for $I =(x^2, xy-zx)$. We will consider \textbf{DEGLEX} ordering,\hspace{1mm}so we consider $x > y > z$. The leading terms of $x^2$ and $xy-zx$ are $x^2$ and $xy$ respectively and both of them cannot be reduced further. So our initial set $G$ is $\{ x^2, xy-zx \}$. However we can compute S-polynomial between them based on $x^2 \cdot y = x \cdot xy$ and obtain
\[ x^2y - x(xy-zx)  \]
\[ = xzx     \]
which cannot be reduced further and so include it in our set $G$. So now our set $G$ is $\{ x^2, xy-zx, xzx \}$. We can compute S-polynomial between $xy-zx$ and $xzx$. We have $ xzx \cdot y = xz \cdot xy$ and so our S-polynomial is 
\[ (xzx)y - xz(xy-zx)  \]
\[ = xzzx  \]
which cannot be reduced further so we add it to $G$ and get $\{ x^2, xy-zx, xzx, xzzx \}$. This is how we proceed and claim that Gr\"{o}bner basis for $I$ is given by
\[ \{x^2, xy-zx \} \cup \hspace{1mm} \bigcup_{i=1}^{\infty}\{ xz^ix \}.  \]
We can prove this using the similar argument we have used in the previous example by method of induction.
\end{example}

\section{Hilbert series and it's computations}

\subsection{Graded algebra}
\begin{definition}
An algebra $A$ is \textbf{graded} if:
\begin{itemize}
\item $ A = \bigoplus_{i=0}^{\infty}A_i$
\item $ A_iA_j \subseteq A_{i+j}$.
\end{itemize}
\end{definition}
\begin{example}
The polynomial algebra is graded by degree.
\end{example}
\subsection{Hilbert Series}
Our main goal is to find tools to investigate properties of algebras. One of the main properties we are interested in is the size of an algebra. Most algebras we will be considering are infinite dimensional, so the question needs to be better defined. In the graded case at least, this is simple - we can ask the dimension of each graded component. This gives a sequence of numbers, which describes the size of the entire algebra. When we have a sequence of numbers, we can talk about the generating function. This is called \textit{Hilbert series} for our graded algebra.
\begin{definition}
Let $A = \bigoplus_{n=0}^{\infty}A_n$ be a graded algebra. Then the formal series
\[ H_A = \sum_{n=0}^{\infty}\dim(A_n)t^n \]
is called the \textbf{Hilbert series} of the algebra $A$.
\end{definition}
\begin{example}
Let us give an example - how to find the Hilbert series of the free associative algebra $A = \mathbb{K}\langle x,y,z\rangle$.\\
One can easily see that this algebra is graded - the graded components are homogeneous subspaces of each degree. So the subspace of degree $0$ has dimension $1$. The subspace of degree $1$ is generated by $x,y,z$. So that has dimension $3$. The subspace of degree $2$ is generated by $x^2,xy,xz,yx,y^2,yz,zx,zy,z^2$, which is of dimension $9$.\\
In general, the subspace of degree $n$ has dimension $3^n$ (as there are $3$ independent choices for each of the $n$ positions). So
\[ H_A = \sum_{n=0}^{\infty}3^nt^n = \frac{1}{1-3t}   \]
\end{example}
\begin{theorem}
If we have two graded algebras $U$ and $V$, then we grade the algebra $U \oplus V$ with components $U_n \oplus V_n$ (except when $n=0$, in which case $(U \oplus V)_n = K$). Similarly we can define the graded components of $U \otimes V$ as $\sum_{i=0}^{n}U_i \otimes V_{n-i}$. Then
\[ H_{U \oplus V} = H_U + H_V  \]
and
\[  H_{U \otimes V} = H_UH_V   \]
\end{theorem}
\begin{proof}
For the first equality, we simply note that:
\[ \dim(U \oplus V)_n = \dim U_n + \dim V_n.   \]
For the second,
\[  \dim(U \otimes V)_n = \sum_{i=0}^{n}\dim U_i \dim V_{n-i}  \]
this exactly matches the co-efficient of $t^n$ in $H_UH_V$.
\end{proof}
\begin{corollary}
We call a subspace $V$ homogeneous if $V = \oplus A_n \cap V$. If $U$,$V$ are homogeneous subspaces of some algebra, then $H_{U+V} \leq H_U + H_V$ and $H_{UV} \leq H_UH_V$ (where the inequality is co-efficient wise, and the sum and product are set sum and multiplications).
\end{corollary}
\begin{proof}
This directly follows from theorem 7.2.1. We have equality if every element of $H_{U+V}$ can be written uniquely as $u+v$.
\end{proof}
\begin{lemma}
If $A$ is an algebra with set of normal words $N$, then $H_A = H_N$.
\end{lemma}
\begin{proof}
$A$ is the free associative algebra, mod some ideal $I$. The free associative algebra itself decomposes into $N \oplus I$, so it's Hilbert series can be written as $H_N + H_I$. Modding out by $I$ turns the free algebra into our algebra $A$, and kills the $H_I$ term. So $H_A = H_N$ as required.
\end{proof}
\begin{example}
(Hilbert series of polynomial algebra):\\
Let us give another example - how to find the Hilbert series of the polynomial algebra $A = \mathbb{K}[x,y]$.\\
Again we can take grading by degree. So the subspace of degree $0$ has dimension $1$. The subspace of degree $1$ is generated by $x,y$, so it has dimension $2$. The subspace of degree $2$ is generated by $x^2,xy,y^2$, which has dimension $3$.\\
In general, the subspace of degree $n$ has dimension $n+1$ (as the word is determined by the position where the $x$'s stop and the $y$'s begin; there are $n+1$ such places). So
\[  H_A = \sum_{n=0}^{\infty}(n+1)t^n  \]
we note that
\[  \frac{1}{1-t} = \sum_{n=0}^{\infty}t^n.  \]
Taking derivatives on both sides give:
\[ \frac{1}{(1-t)^2} = \sum_{n=0}^{\infty}(n+1)t^n    \]
So
\[  H_A = \frac{1}{(1-t)^2}  \]
\end{example}
\begin{theorem}
The Hilbert series of the polynomial algebra $\mathbb{K}[X]$ is computed by the formula
\[ H_{\mathbb{K}[X]} = \prod_{x \in X}(1- t^{|x|})^{-1}.   \]
The Hilbert series of the exterior algebra $\bigwedge \mathbb{K}[X]$ is calculated by the formula
\[ H_{\bigwedge \mathbb{K}[X]} = \prod_{x \in X}(1 + t^{|x|}).   \]
In particular, in case of natural graduation and a finite set of generators $d$, we have:
\[ H_{\mathbb{K}[X]}^{-1} = (1-t)^d;\hspace{5mm} H_{\bigwedge \mathbb{K}[X]} = (1+t)^d     \]
\end{theorem}
\begin{proof}
In the case of one generator, the Hilbert series in the power of $n = |x|$ is computed straightforwardly: it is equal to $1 + t^n + t^{2n} + \dots = (1-t^n)^{-1}$ in case of polynomial ring and $1+t^n$ in the case of exterior algebra. The case of finite number of generators reduces to this one, with the help of theorem 7.2.1. Finally, in case of infinite number of generators, the degree of generators must increase, for if not, we do not get finite-dimensionality. Consequently, for every $n$, the segment of the Hilbert series up to the exponent $n$ depends only on finite number of generators with the degree not exceeding $n$, thus everything reduces to the finite case.
\end{proof}
\begin{remark}
If we apply the formula of Hilbert series for polynomial algebra of theorem 7.2.2 to our previous computed example 7.2.2 we will get the exact result.
\end{remark}
\subsection{Hilbert series for free product} 
\begin{definition}
We define the \textbf{free product} of two algebra $A,B$ as the disjoint union of their generators, with both sets of relations. We usually denote it by $A \ast B$.
\end{definition}
\begin{example}
The free product of the algebra $A = \langle x\hspace{1mm} |\hspace{1mm} x^3 + 2x^2 \rangle$ and the algebra $B = \langle x,y\hspace{1mm} |\hspace{1mm} 2x^2 = y^2 \rangle$ is given by the algebra $ A \ast B = \langle x,y,z\hspace{1mm} |\hspace{1mm} x^3 + 2x^2 , 2y^2 = z^2 \rangle$.
\end{example}
\begin{theorem}[Hilbert series of free product]
If $A,B$ are graded algebras, then
\[ (H_{A \ast B})^{-1} = H_A^{-1} + H_B^{-1} - 1.   \]
\end{theorem}
\begin{proof}
Any word/monomial in $A \ast B$ is either begins with a (non-scalar) element of $A$, or an element of $B$ (excluding terms that belong to the underlying field). This follows from the fact that, because the two sets of generators have no overlap, the Gr\"{o}bner basis of $A \ast B$ will be the union of the Gr\"{o}bner basis of $A$ and the Gr\"{o}bner basis of $B$. So take any monomial in $A \ast B$. It begins with a generator either from $A$ or $B$. Without loss of generality say its from $A$. Then take the longest prefix of this word that consists of generators from $A$. This must be a word in $A$, otherwise it would contain a leading term of the Gr\"{o}bner basis for $A$, and hence a leading term in the Gr\"{o}bner basis for $A \ast B$.\\ 
So we know that words in $A \ast B$ start with a word from either $A$ or $B$. That suggests the following decomposition:
\[  H_{A \ast B} = H_{V_1} + H_{V_2} +1     \]
where $V_1$ are the elements that begin with a non-scalar element of $A$, and $V_2$ are the elements that begin with a non-scalar element of $B$. So we have
\[  H_{V_1} = (H_A -1)(H_{V_2} + 1)    \]
\[  H_{V_2} = (H_B -1)(H_{V_1} + 1).     \]
The above imply that
\[  H_{V_1}  = \frac{H_AH_B - H_B}{H_A + H_B - H_AH_B}    \]
\[  H_{V_2}  = \frac{H_AH_B - H_A}{H_A + H_B - H_AH_B}    \]
So
\[  H_{A \ast B} = \frac{H_AH_B - H_B}{H_A + H_B - H_AH_B} + \frac{H_AH_B - H_A}{H_A + H_B - H_AH_B} + 1      \]
\[  = \frac{2H_AH_B - H_A - H_B - H_AH_B + H_A + H_B}{H_A + H_B - H_AH_B}    \]
\[  = \frac{H_AH_B}{H_A + H_B - H_AH_B}.     \]
Hence
\[  (H_{A \ast B})^{-1} = H_A^{-1} + H_B^{-1} -1.   \]
\end{proof}

\subsection{Calculating Hilbert series through chains}
Let $A = \mathbb{K}\langle X \hspace{1mm}|\hspace{1mm}R \rangle$ be an algebra where $X$ is the set of generators and $R$ is the set of relations. Let the set of normal words be $N$, and say we have a reduced Gr\"{o}bner basis $G$. Let $F$ be the set of leading terms of $G$ - we call the elements of $F$ \textit{obstructions}. A word is normal with respect to $G$ if and only if it does not contain any of the elements of $F$ as a subword. Thus a word $s$ is normal with respect to $G$ if and only if it's normal with respect to $F$. Therefore the algebra $\hat{A} = \mathbb{K}\langle X \hspace{1mm}|\hspace{1mm}F \rangle$ has the same normal words as $A$. And so $H_A = H_{\hat{A}}$. The advantage of dealing with $F$ instead of $G$ is that $F$ consists of monomials only, which makes it much easier to deal with.

\begin{definition}[Chains]
A $(-1)$ - chain is the empty word and is its own tail. The $0$-chains are the elements of the generating set $X$, and are also their own tails. We will define chain inductively: a $n$-chain is a word $f$ of the form $gt$, with some conditions on $g$ and $t$. Firstly, $g$ must be a $(n-1)$-chain and $t$ is a normal word. Secondly, if $r$ is the tail of $g$ then $\deg_F rt = 1$; that is, the word $rt$ contains exactly one element of $F$ as a subword. This subword must occur at the end of $rt$. The tail of $gt$ is defined to be $t$.\\ 
We denote the space spanned by $n$-chains by $C_n$.
\end{definition}
\begin{example}
Let $F = \{ x^3 \}$. The unique 1-chain is $x^3= x \cdot x^2$ and its tail is $x^2$. Then the unique 2-chain is $\overline{xxxx} = x^3 \cdot x$. The word $x^3 \cdot x^2$ is not a 2-chain, since $\deg_F x^2x^2 = 2$. The unique 3-chain is the word $x^6=x^4x^2$. The word $x^5 = x^4x$ is not a 3-chain because $\deg_F x\cdot x =0$, regardless of the fact that it can be represented ($xxxxx$) as a link of three obstructions $\overline{xxx}$ (the fact is that the first one intersects with the last one).\\ 
In general the $n$-chain is given by $x^{n+1}\cdot x$ if $n$ is even and $x^{n+1}\cdot x^2$ if $n$ is odd. We see that in this case for every $n$ there exists only one $n$-chain.
\end{example}
\begin{theorem}
Let
\[ \begin{tikzcd}
  \dots \arrow{r}{d_{n+1}} & A_n \arrow{r}{d_n} & A_{n-1} \arrow{r}{d_{n-1}} & \dots \arrow{r}{d_{k+1}} & A_k \arrow{r}{d_k} & K \arrow{r} & 0 
  \end{tikzcd}\]
  be an exact sequence of graded spaces (i.e $\ker d_i = \Ima d_{i+1}$). Then
  \[  \sum_{i=k}^{\infty} (-1)^i H_{A_i} = (-1)^k     \]
  if the sum is well defined.
\end{theorem}
\begin{proof}
We know that if $f: V \rightarrow W$ is any linear transformation such that $f(V_n) \subset f(W_n)$, then $H_V = H_{\ker f} + H_{\Ima f}$. So
\[  H_{A_i} = H_{\ker {d_i}} + H_{\Ima {d_i}} = H_{\Ima {d_{i+1}}} + H_{\Ima {d_i}},  \]
where the second equality comes from exactness. Thus taking the alternating sum of these equalities, we get our required result.
\end{proof}
\begin{theorem}[Hilbert series from chains]
Let $A$ is an algebra and $C_n$ be the linear span of $n$-chains. We point out that
\[ H_{C_{-1}} = H_K = 1;\hspace{4mm} H_{C_{0}} = H_X; \hspace{4mm} H_{C_1}= H_F.   \]
Then
\[  H_A = ( H_{C_{-1}} - H_{C_0} + H_{C_1} - H_{C_2} + \dots )^{-1}.    \]
\end{theorem}
\begin{proof}
Let $\hat{A} = \mathbb{K}\langle X \hspace{1mm}|\hspace{1mm}F \rangle$ where $F$ is our set of obstructions. Let $d_n: C_n \otimes \hat{A} \rightarrow C_{n-1} \otimes \hat{A}$ be defined by $d_n(gt \otimes a) = g \otimes ta$. Then $d_n(d_{n+1}(gt \otimes a)) = d_n( g \otimes ta) = g' \otimes t'ta$, where $g \in C_n$, $g' \in C_{n-1}$. By the definition of $n$-chains it follows that there must be an obstruction in $t'ta$, so this is $0$.\\ 
Also $d_n$ is a surjection, as for every $ g \otimes a \in C_{n-1} \otimes \hat{A}$, we can decompose $a$ as $a_1a_2$ such that $ga_1 \otimes a_2 \in C_n \otimes \hat{A}$. So we have an exact sequence:
\[ \begin{tikzcd}
 \dots \arrow{r}{d_{n+1}} & C_n \otimes \hat{A} \arrow{r}{d_n} & C_{n-1} \otimes \hat{A} \arrow{r}{d_{n-1}} & \dots \arrow{r}{d_0} & C_{-1} \otimes \hat{A} \arrow{r} & K \arrow{r} & 0
\end{tikzcd}\]
So by theorem 7.4.1
\[  \sum_{i=-1}^{\infty} (-1)^i H_{C_i \otimes \hat{A}} = -1.   \]
Now $H_{C_i \otimes \hat{A}} = H_{C_i}H_{\hat{A}} = H_{C_i}H_A$. Hence
\[  H_A = ( H_{C_{-1}} - H_{C_0} + H_{C_1} - H_{C_2} + \dots )^{-1}.    \]
\end{proof}
\begin{example}
Let $A = \mathbb{K}\langle x,y \hspace{1mm}|\hspace{1mm}x^2+y^2\rangle$. The obstructions of $A$ are $x^2$ and $xy^2$ (the Gr\"{o}bner basis of $A$ is $\{ x^2+y^2, xy^2 - y^2x \}$). We have found that the $n$-chains are given by $x^ny^2$ and $x^{n+1}$, for $n > 0$. So we have,
\[ H_A^{-1} = ( 1- 2t +(t^2+t^3) -(t^3+t^4) + \dots ) = 1-2t+t^2.    \]
Hence
\[  H_A = \frac{1}{1-2t+t^2}.   \]
\end{example}

\section{Conclusion and Further work}
 In the proof of theorem 7.4.2, we have constructed a resolution for the normal words of our algebra. This allows us to find the Hilbert series for our normal words and hence for our original algebra. However this resolution only depends on normal words, so the rest of the structure of our algebra is lost. This is okay if we are only interested in the size of algebra. If we want more properties of the algebra then this leads to a construction of a resolution of the algebra itself. Along with this fact, that sequence and isomorphism of spaces with graduations $C_n \otimes \hat{A}$ and $C_n \otimes A$ leads us to think about the existence of a corresponding free resolution which was constructed by Anick. We refer [3] as a good reference to study Anick's resolution to the interested readers. 

\section{References}
\begin{enumerate}
\item Bruno Buchberger, An Algorithm for Finding the Basis Elements of the Residue Class Ring of a Zero Dimensional Polynomial Ideal, Ph.D. Dissertation,1965. \textit{Journal of Symbolic Computation}, \textbf{41}, 2006 (Translation).
\item Viviana Ene - J\"{u}rgen Herzog, \textit{Gr\"{o}bner Bases in Commutative Algebra}, Graduate Studies in Mathematics, Volume 130, AMS.
\item V.A.Ulfarovskij, \textit{ Combinatorial and Asymptotic Methods in Algebra}. pp.- 42-58. Algebra VI, Springer, 1995.
\end{enumerate}

\Addresses
\end{document}